\documentclass[11pt,a4paper, twoside, leqno]{amsart}
\usepackage[shortalphabetic, nobysame]{amsrefs}
\usepackage[english]{babel}
\usepackage[utf8]{inputenc}

\usepackage{indentfirst}
\usepackage{graphicx}
\usepackage{newlfont}
\usepackage[active]{srcltx}
\usepackage{mathrsfs}
\usepackage{amssymb,amscd,setspace, wasysym}
\usepackage{amsmath}
\usepackage{latexsym}
\usepackage{amsthm}
\usepackage[all]{xy}
\usepackage{comandi}
\usepackage[a4paper,left=3cm,right=3cm,top=4cm,bottom=4cm]{geometry }
\usepackage{hyperref}
\hypersetup{
 colorlinks=true,
 citecolor=cyan,
 linkcolor=blue,
 urlcolor=cyan}%make all the references and links clickable
\numberwithin{equation}{section}
\newcommand\rp{{\mathbf{P}}}

\newcommand{\rk}{\mathrm{rk}}

\theoremstyle{plain}

\title[$GV$-subschemes and their embeddings]{\textit{\bfseries GV}-subschemes and their 
embeddings in principally polarized abelian varieties}

\author{Luigi Lombardi and Sofia Tirabassi}
\address{Mathematical Institut\\ University of Bonn, Endenicher Allee 60, 53115, Germany}
 \email{\url{lombardi@math.uni-bonn.de}}

\address{Department of Mathematics\\ The University of Utah, 155 S 1400 E, JWB 233, Salt Lake City, UT 84112-0090, USA}
\email{\url{sofia@math.utah.edu}}

\subjclass[2010]{}

\date{}

\begin{document}

\begin{abstract}
We prove that the embedding of a $GV$-subscheme in a principally polarized abelian variety does not factor through any nontrivial isogeny. 
As an application, we present 
a new proof of a theorem of Clemens--Griffiths identifying the intermediate Jacobian of 
a smooth cubic threefold in $\rp ^4$ to the Albanese variety of its Fano surface of lines.
\end{abstract}

\maketitle

\vspace{-.5cm}\section{Introduction}
Given a principally polarized abelian variety $(A, \Theta)$ with $\dim A=g$, we say that
a reduced subscheme $X\subseteq A$ of dimension $d$ has
\emph{minimal cohomology class} if $[X]=\frac{[\Theta]^{g-d}}{(g-d)!}$ in $H^{2g-2d}(A,\mathbf{Z})$.
Beside the theta divisor of $A$ itself, at the moment there are only two known families of subvarieties having minimal cohomology class. 
These are the Abel--Jacobi images of a symmetric product of a curve into its Jacobian
and the Fano surfaces of lines embedded in intermediate Jacobians attached to smooth cubic threefolds in $\rp^4$ (see \cite{BL} and \cite{CG1972}).
Furthermore, Debarre in \cite{De1995} conjectures that, 
extending in higher dimension questions of Beauville and Ran, these are the only possible examples when $A$ is indecomposable.

The Matsusaka--Ran criterion yields this conjecture in the case of curves. 
Moreover, as shown in \cite{De1995}*{Theorem 5.1}, the only reduced subschemes of a Jacobian having minimal cohomology class are 
the aforementioned Brill--Noether loci. 
Finally, in \cite{Ho2010}, H\"oring proves that if $(A, \Theta)$ is a general intermediate Jacobian of a smooth cubic threefold in $\rp^4$, 
then the Fano surface of lines is the only subscheme having minimal cohomology class, with the exception, of course, of the theta divisor.

In \cite{PP2008} the authors propose to attack Debarre's conjecture by studying a 
cohomological condition on the twisted ideal sheaf $\fas I_X(\Theta)$ of a reduced closed subscheme $X$ in a principally polarized abelian variety $(A,\Theta)$, 
by means of Fourier--Mukai transforms. 
More precisely we say that a subscheme $X$ embedded in $(A,\Theta)$ is a $GV$-\emph{subscheme} if 
$$\codim _{\widehat{A}}\,\big\{\alpha \in \widehat{A}\, | \, h^i(A,\fas I_X(\Theta))>0\big\}\geq i\quad \mbox{ for all }\quad i>0$$
(\emph{cf}. Definition \ref{GVsub}). One of the main results of \cite{PP2008} is that
geometrically nondegenerate\footnote{A $d$-dimensional subvariety $X$ of an abelian variety $A$ is geometrically nondegenerate if 
the kernel of the restriction map $H^0(A,\Omega_A^d)\rightarrow H^0(X_{reg},\Omega^d_{X_{reg}})$ contains no nonzero decomposable $d$-forms.} 
$GV$-subschemes have minimal cohomology class. Furthermore, Pareschi--Popa also prove that 
the only geometrically nondegenerate $GV$-subschemes in $A$ of dimensions $1$ and $g-2$ 
are the Abel--Jacobi images of a curve in its Jacobian and the Brill--Noether loci $W_{g-2}$ respectively.

In this notes we derive some further geometric properties of $GV$-subschemes. 
Our main result is the following:
\begin{thm}\label{thmintr}
Let $X$ be a reduced closed $GV$-subscheme of a principally polarized abelian variety $(A, \Theta)$ and let $\iota:X\hookrightarrow A$ be the closed immersion. 
\begin{enumerate}
\item If $X$ generates $A$ as a group, then the immersion $\iota:X\hookrightarrow A$ does not factor through any nontrivial isogeny.
\item If $X$ is of positive dimension and generates a proper abelian subvariety $J$ of $A$, then $A$ is a product 
of principally polarized abelian varieties having $J$ as one of the factors. Furthermore, $X$ is a $GV$-subscheme in $J$.
\end{enumerate}

\end{thm}

In order to prove Theorem \ref{thmintr} we will use the theory of $GV$ and $M$-\emph{regular sheaves} 
on abelian varieties as described by Pareschi--Popa in \cites{PP2011,PP2011b} (\emph{cf}. Definition \ref{def:M}). 
% Moreover, Pareschi--Popa established a criterion to see whether or not a sheaf is either $GV$ or $M$-regular by means of Fourier--Mukai transforms; this will be 
% one of the main tools in our proofs (see Theorem \ref{thm:alg}). 
We will proceed as follows. First of all we show that in case the embedding 
$\iota$ factors through a nontrivial isogeny $\varphi:B\rightarrow A$, then $\fas{I}_{X/B}\otimes \varphi^* \fas{O}_A(\Theta)$, the ideal sheaf 
of $X$ in $B$ twisted by the pull-back of $\Theta$, is a $GV$-sheaf on $B$. 
By using a characterization of $GV$-sheaves this means that the Fourier-Mukai transform of the dual of $\fas{I}_{X/B}\otimes \varphi^* \fas{O}_A(\Theta)$
can be written as a tensor product of an ideal sheaf of a subscheme of codimension at least two and a
line bundle $M$ on $\widehat B$.
Finally, the last step is to deduce positivity properties on $M$. In particular, by using our generation hypothesis on $X$, 
we show that the line bundle $M$ is ample on $\widehat B$ and moreover that the pull-back $\varphi ^*M$ is a principal polarization.
This immediately yields that $\deg \varphi = h^0(B,\varphi^*M)=1$.

In the last section we give an application of Theorem \ref{thmintr} by providing a new proof of a theorem of Clemens--Griffiths stating that 
the intermediate Jacobian of a smooth cubic threefold in $\rp^4$ is isomorphic to the Albanese variety of its Fano surface of lines 
(\cite{CG1972}*{Theorems 11.19}). Besides Theorem \ref{thmintr} this proof also relies 
on a theorem of H\"{o}ring showing that the Fano surface of lines of a smooth cubic threefold in $\rp^4$ is a $GV$-subscheme (\cite{Ho2007}*{Theorem 1.2}).
% Besides Theorem \ref{thmintr}
% this new proof also relies on a computation carried out by H\"{o}ring showing that the Fano surface of lines $F\subset J(Y)$ is a $GV$-subscheme (\cite{Ho2007}).

\subsection*{Notation}
Throughout this paper we work over an algebraically closed field $K$ of characteristic zero unless otherwise specified.
Given a smooth variety $Z$ over $K$, we will denote by $\mathbf{D}(Z)$ 
the derived category of bounded complexes of quasi-coherent sheaves on $Z$ having coherent cohomology.
If $A$ is an abelian variety, we denote by $0_A$ the neutral element of $A$, by $(-1)_A$ the multiplication by $-1$, 
and by $\widehat A$ the dual abelian variety. Moreover, if 
$\Theta\subseteq A$ is a principal polarization, we denote by $\widehat \Theta\subset \widehat A$ the dual polarization.
% also the abelian variety dual to $A$, that we will denote by $\hat A$ is principally polarizable. We will use the symbol $\hat\Theta$ to indicate a principal polarization on $\hat A$.

\subsection*{Acknowledgments:} We are grateful to Giuseppe Pareschi and Mihnea Popa for introducing us to the subject of minimal cohomology classes. Moreover  we thank
Herb Clemens, Giovanni Cerulli Irelli, Christopher Hacon, Daniel Huybrechts, Elham Izadi and Christian Schnell for several useful conversations and for answering to all our
questions. Finally, special thanks go to Andreas H\"{o}ring who spot a mistake in the first draft of this work.

This project was started when ST, supported by the AWM-NSF Mentoring Grant (NSF award number DMS-0839954), 
was visiting the University of Illinois at Chicago (UIC) while LL was a graduate student there. 
We are grateful to UIC for the nice working environment and the kind hospitality. 
LL was supported by the SFB/TR45 ``Periods, moduli spaces, and arithmetic of algebraic varieties'' of the DFG (German Research Foundation).

\section{Background Material}
\subsection{Fourier--Mukai Transforms}\hfill\par
One of the main tools we will use in this paper 
is the Fourier--Mukai transform between the derived category of an abelian variety $A$ of dimension $g$ and the derived category of its dual $\widehat A$. This is defined as 
follows.
Let $\mathscr P$ be a normalized Poincar\'e line bundle on $A\times \widehat A$, then we define the functor
\[\mathbf R S_A:\mathbf D(A)\longrightarrow \mathbf D(\widehat A)\]
by setting
\[\mathbf RS_A(-)=\mathbf Rp_{\widehat A\:*}(p_A^*(-)\otimes \mathscr P)\]
where $p_A$ and $p_{\widehat A}$ are the first and second projections from the product $A\times \widehat A$, respectively. 
Similarly we can also consider the Fourier--Mukai transform in the other direction, again induced by the Poincar\'e line bundle $\mathscr P$
\[\mathbf R \widehat S_A:\mathbf D(\widehat A)\longrightarrow \mathbf D(A)\]
by setting
\[\mathbf R\widehat S_A(-)=\mathbf Rp_{ A\:*}(p_{\widehat A}^*(-)\otimes \mathscr P).\]
Mukai's inversion theorem \cite{Mu1981}*{Theorem 2.2} tells us that these functors are
equivalences of triangulated categories. More precisely the following formulas hold:
\begin{equation}\label{Mukai duality}
\mathbf R S_A\circ \mathbf R\widehat S_A\simeq (-1_{\widehat{A}})^*\circ[-g],\quad
\mathbf R\widehat S_A\circ\mathbf R S_A\simeq(-1_{A})^*\circ[-g]; 
\end{equation}
where $[-]$ stands for, as usual, the shift functor in a triangulated category. 

For an arbitrary abelian variety $A$ we denote by 
$$\mathbf R \Delta_A: \mathbf D(A)\rightarrow \mathbf D(A),\quad  \fas F\mapsto \mathbf R\mathcal{H}om_A (\fas F,\fas O_A)$$ the derived dual functor.
The commutativity between the derived dual and the Fourier--Mukai transforms $\rr S_A$ and $\rr \widehat S_A$
is described by the following formulas, again due to Mukai (\cite{Mu1981}*{(3.8) and Theorem 3.13}):
\begin{equation}\label{Mukai dual}
 \mathbf R \Delta_{\widehat{A}} \circ \mathbf R S_A\simeq ((-1_{\widehat{A}})^*\circ \mathbf R S_A \circ \mathbf R \Delta_A)[g],\quad  
 \mathbf R \Delta_A \circ \mathbf R \widehat{S}_{A} \simeq ((-1_{A})^*\circ \mathbf R \widehat{S}_{A}
 \circ \mathbf R \Delta_{\widehat{A}})[g].
\end{equation}

\subsection{{\itshape GV}-subschemes}
In this subsection we recall a few facts regarding $GV$-\emph{sheaves} and introduce the main characters of this work, i.e. $GV$-\emph{subschemes}.

Let $A$ be  
an abelian variety of dimension $g$ and let $\mathscr F$ be a coherent sheaf on $A$. For any $i\ge 0$ we define the 
\emph{i-th cohomological support locus of} $\fas F$ as 
\[V^i_A(\mathscr F):=\{\alpha\in \widehat A\:|\: h^i(A, \mathscr F\otimes \alpha)> 0\}.\]
\begin{dfn}
 Given an integer $k\geq 0$, we say that $\mathscr F$ is a $GV_k$-sheaf  
 if one of the following equivalent conditions holds:\footnote{A proof of their equivalence is in \cite{PP2011}*{Lemma 3.6}.} 
\begin{enumerate}
 \item $\mathrm{codim}_{\widehat A}\, V^i_A(\mathscr F)\ge i+k$ for every $i>0$\\
\item $\mathrm{codim}_{\widehat A}\, \mathrm{Supp}\, R^i S_A(\mathscr F)\ge i+k$ for every $i>0$.
\end{enumerate}

\end{dfn}
Usually $GV_0$-sheaves are simply called $GV$, while $GV_1$-sheaves are said $M$-\emph{regular sheaves}. Also note that $M$-regular sheaves are $GV$.
We highlight the following result giving equivalent conditions for a sheaf to be either $GV$ or $M$-regular
 (we refer to \cite{PP2011}*{\cite{PP2011b}} for further generalities regarding $GV_k$-sheaves). 

\begin{thm}[{\cite{PP2011b}*{Theorem 2.3 and Proposition 2.8}}]\label{thm:alg}
 Let $A$ be an abelian variety of dimension $g$  
over an algebraically closed field and let $\fas{F}$ be a coherent sheaf on $A$. 
\begin{enumerate}
\item The sheaf $\fas{F}$ is $GV$ if and only if the cohomology sheaves
$R^iS_A\mathbf R\Delta_A(\fas{F})$ vanish for every $i\neq g$.  Furthermore, if this is the case, then we have $\rk \,R^g S_A\mathbf R\Delta_A(\fas{F}) = \chi(A,\fas{F})$.\\
\item If $\fas{F}$ is a $GV$-sheaf, then it is $M$-regular if and only if $R^g S_A\mathbf R\Delta_A(\fas{F})$ is a torsion free sheaf.
\end{enumerate}
\end{thm}

\begin{dfn}\label{GVsub}
We say that a reduced closed subscheme $X$ of a principally polarized abelian variety $(A, \Theta)$ defined by an ideal sheaf 
$\fas{I}_X\subset \fas{O}_A$ is a $GV$-\emph{subscheme} if $\fas{I}_X(\Theta):=\fas{I}_X\otimes \fas{O}_A(\Theta)$ is a $GV$-sheaf.
\end{dfn}

A useful criterion to detect $GV$-subschemes is the following:

\begin{prop}[{\cite{PP2008}*{Lemma 3.3}}]\label{lem:trfonA}
 Let  $\iota:X\hookrightarrow A$ be a reduced closed subscheme of a principally polarized abelian variety $(A, \Theta)$ of dimension $g$. \
Then $X$ is a $GV$-subscheme if and only if $\iota_*\OO{X}(\Theta)$ is an $M$-regular sheaf with $\chi(A,\iota_*\OO{X}(\Theta))=1$. In particular, if $X$ is a 
$GV$-subscheme, then there is an isomorphism of complexes
 \begin{equation}\label{eq:L}
  \mathbf{R}S_A\mathbf{R}\Delta_A(\iota_*\OO{X}(\Theta))\simeq \fas{I}_Z\otimes L[-g],
 \end{equation}
where $\fas{I}_Z\subset \OO{\widehat A}$ is an ideal sheaf whose zero locus does not contain any divisorial components and $L$ is the reflexive hull of $R^g
S_A\mathbf{R}\Delta_A(\iota_*\OO{X}(\Theta))$.
\end{prop}

\begin{rmk}\label{rem:prod}
With notation as in the previous proposition, we observe that from Proposition \ref{lem:trfonA} and \cite{Mu1981}*{Theorem 3.13 (5)}
there is an inclusion of sheaves $\fas{I}_Z\otimes L\hookrightarrow \OO{\widehat A}(\widehat \Theta)$ 
obtained by applying the functor $\rr S_A\rr \Delta_A$ to the exact sequence 
$$0\longrightarrow \fas{I}_X(\Theta)\longrightarrow \fas{O}_A(\Theta)\longrightarrow \iota_*\OO{X}(\Theta)\longrightarrow 0.$$ 
Moreover, by taking $\fas{H}om$'s,  we get a further inclusion
 $$\OO{\widehat A}(-\widehat \Theta)\hookrightarrow L^{-1}$$ so that 
 line bundle $\OO{\widehat A}(\widehat\Theta)\otimes L^{-1}$ has one nonzero global section
which can be written as $\OO{\widehat A}(E)$ for some effective divisor $E$ on $\widehat A$. 
Suppose now that $L$ itself has one nonzero global section so that $L\simeq\OO{\widehat A} (E')$ for some effective divisor $E'$, 
then only two possibilities occur:
 \begin{enumerate}
  \item either $L$ is ample, and therefore $E=0$ and $E'\simeq \Theta$;\\
  \item or $L$ is not ample. In this case we have then that the divisor $\widehat \Theta = E+E'$ is reducible
  as $\widehat \Theta$ is a principal polarization, and hence, by the Decomposition Theorem in \cite{BL}*{Theorem 4.3.1},
$A$ is a product of nontrivial principally polarized abelian varieties.
 \end{enumerate}
\end{rmk}

\section{Proof of the main theorem}
The results stated in the Introduction are consequences of the following technical statement:
\begin{prop}\label{prop:tehc}
 Let $\iota:X\hookrightarrow A$ be a $GV$-subscheme of a principally polarized abelian variety $(A, \Theta)$ and suppose that the 
inclusion $\iota:X\hookrightarrow A$ factors through a nontrivial isogeny $\varphi$ as in the diagram below
 \[
 \xymatrix{X\ar[rr]^a\ar[dr]_i&&B\ar[dl]^\varphi\\
 &A.&}
\]
Then there is an isomorphism of complexes of sheaves
$$\mathbf R S_B\mathbf{R}\Delta_B(a_*\OO{X}\otimes\varphi^*\OO{A}(\Theta))\simeq \fas{I}_W\otimes M[-g]$$ where 
$\fas{I}_W$ is an ideal sheaf on $\widehat B$ whose zero locus does not contain any divisorial components, 
$M$ is a line bundle on $\widehat B$ such that $\widehat\varphi^*M\simeq L$, and $L$ is the line bundle appearing in \eqref{eq:L}.
\end{prop}
\begin{proof}
 In \cite{Mu1981}*{(3.4) p. 159} Mukai studied the behavior of the Fourier--Mukai transform with respect to an arbitrary isogeny. 
Using his result, together to Proposition \ref{lem:trfonA}, we get the following chain of isomorphisms in the derived category $\mathbf{D}(\widehat A)$:
 \begin{eqnarray}\notag
  \fas{I}_Z\otimes L[-g]&\simeq&\mathbf R S_A\mathbf R \Delta_A(\iota_*\OO{X}(\Theta))\notag\\
  & \simeq &\mathbf R S_A\mathbf R \Delta_A(\varphi_*a_*\OO{X}(\Theta)) \notag\\
  & \simeq & \mathbf R S_A\mathbf{R} \Delta_A(\varphi_*(a_*\OO{X}\otimes\varphi^*\OO{A}(\Theta)))\notag\\
  &\simeq &\widehat \varphi^*\mathbf R S_B\mathbf R \Delta_B(a_*\OO{X}\otimes\varphi^*\OO{A}(\Theta)).\notag
 \end{eqnarray}
 Therefore, by using the fact that $\varphi$ is a flat morphism, we have that 
$$\widehat\varphi^*(R^i S_B\mathbf{R}\Delta_B(a_*\OO{X}\otimes\varphi^*\OO{A}(\Theta))=0\quad \mbox{ for }\quad i<g.$$ From this it follows easily that 
 $$R^i S_B\mathbf{R}\Delta_B(a_*\OO{X}\otimes\varphi^*\OO{A}(\Theta))=0 \quad\text{for $i<g$},$$
 and therefore $$\mathbf{R} S_B\mathbf R \Delta_B(a_*\OO{X}\otimes\varphi^*\OO{A}(\Theta))\simeq \fas{F}[-g]$$ for some coherent sheaf $\fas{F}$ on $\widehat B$.\par
 By shifting complexes to the left, we get an isomorphism of sheaves
 $$\fas{I}_Z\otimes L\simeq \widehat\varphi^*{\fas{F}}$$
 from which we can deduce that $\fas{F}$ has generic rank equals to $1$. 
Furthermore, we get an inclusion of sheaves
 $$\fas{F}\hookrightarrow\widehat\varphi_*(\widehat \varphi^*\fas{F})\simeq\widehat\varphi_*(\fas{I}_Z\otimes L),$$
from which we notice that since the latter is torsion free, then so is $\fas{F}$. Therefore we can write
 $$\fas{F}\simeq\fas{I}_W\otimes M$$
where $M$ is the reflexive hull of $\fas{F}$ and $\fas{I}_W$ is an ideal sheaf whose zero locus does not contain any divisorial components.
 Finally, by applying $\fas{H} om_{\OO{\widehat A}}(-,\OO{\widehat A})$ 
to the isomorphism $\widehat\varphi^*(\fas{I}_W\otimes M)\simeq\fas{I}_Z\otimes L$,
 we get $(\widehat \varphi^* M)^{-1}\simeq L^{-1}$. 
As a consequence, we have $\widehat\varphi^*M\simeq L$ and the statement is proved.
 \end{proof}

 We now divide the proof of Theorem \ref{thmintr} into two cases: when $X$ generates $A$, and when it does not.
\subsection{Case 1: When $X$ generates $A$ (proof of Theorem \ref{thmintr} (1))}\hfill\par

Suppose that the hypotheses of Theorem \ref{thmintr} (1) hold, i.e. $\iota:X\hookrightarrow (A,\Theta)$ is an embedding of a reduced $GV$-subscheme in a $g$-dimensional 
principally polarized abelian variety such 
that $X$ generates $A$ as a group. We are going to show that the inclusion $\iota:X\hookrightarrow A$ does not factor through any nontrivial isogeny. 
By arguing by contradiction, if it did, then we would be in the same situation of diagram \eqref{eq:factors} under the extra information that $a(X)$ spans $B$. 
Then, by using Proposition \ref{prop:tehc} we can conclude that
\begin{equation}\label{def:M}  \mathbf{R}S_B\mathbf R\Delta_B(a_*\OO{X}\otimes \varphi^*\OO{A}(\Theta))\simeq \fas{I}_W\otimes M[-g].
\end{equation}
The key point of our argument is the following general:

\begin{lem}\label{lem:mample}
Let $\fas F$ be an $M$-regular sheaf on an abelian variety $A$ of dimension $g$ whose support spans $A$. If $\chi(A,\fas{F})=1$, 
then the reflexive hull of $R^g S_A\rd_A (\fas{F})$ is an ample line bundle on $\widehat A$.
\end{lem}

\begin{proof}
Denote by $L$ the reflexive hull of $R^g S_A\rd_A (\fas{F})$. By Theorem \ref{thm:alg} it is a line bundle on $\widehat{A}$ such that  
$$\mathbf{R}S_A\mathbf R \Delta_A(\fas{F})\simeq \fas{I}\otimes L[-g]$$
for some ideal sheaf $\fas{I}$ on $\widehat A$.
Moreover, by applying Mukai's inversion theorem we get an isomorphism
\begin{equation}\label{eq:formula}
 \mathbf{R} \widehat S_A\mathbf{R}\Delta_{\widehat{A}}(\fas{I}\otimes L)\simeq \fas{F}[-g].
\end{equation}

 Let $\fas P$ be a Poincar\'{e} line bundle on $\widehat A\times A$ and denote by $\fas P_x$ the restriction $\fas P_{|\widehat{A}\times \{x\}}$.
 If by contradiction $L$ were not ample, then there would exist a proper abelian subvariety 
 $G\subsetneq A$ such that $H^0(\widehat A, L\otimes\fas P_{x})=0$ for any $x\notin G$. Therefore
\begin{eqnarray}\notag
 0 & = & H^0(\widehat A,\fas{I}\otimes L\otimes\fas P_{x})\\\notag
& \simeq & \mathrm{Hom}_{\mathbf{D}(\widehat A)}(\OO{\widehat A},\fas{I}\otimes
L\otimes \fas P_{x})\\\notag
&\simeq & \mathrm{Hom}_{\mathbf{D}(\widehat A)}(\rd_{\widehat A}(\fas{I}\otimes L),\fas P_{x}),\notag
\end{eqnarray}
and by applying the Fourier--Mukai transform $\mathbf R \widehat S_A$, 
together to the formulas \eqref{Mukai duality} and \eqref{eq:formula}, we would get
\begin{eqnarray}\notag
0 & = & \mathrm{Hom}_{\mathbf{D}(A)}(\rr \widehat S_A\circ\rd_{\widehat A}(\fas{I}\otimes L),\rr \widehat S_A(\fas P_{x}))\notag\\\notag
& \simeq & \mathrm{Hom}_{\mathbf{D}(A)} (\fas{F}[-g],\fas O_{x}[-g])\label{spiega}\\\notag
&\simeq & \mathrm{Hom}_{\mathbf{D}(A)}(\fas{F},\fas O_{x})
\end{eqnarray}
where $\fas O_{x}$ denotes the skyscraper sheaf at $x$.
This says that the point $x$ does not belong to the support of $\fas{F}$. 
Thus we conclude that $\mathrm{Supp}(\fas{F})\subset G$ which contradicts the hypothesis that $\mathrm{Supp}(\fas{F})$ generates $A$.
\end{proof}

\begin{cor}\label{cor:mample}
 The line bundle $M$ in \eqref{def:M} is ample on $\widehat B$.
\end{cor}
\begin{proof}
From Proposition \ref{prop:tehc} and Theorem \ref{thm:alg} we deduce that $a_*\OO{X}\otimes\varphi^*\OO{A}(\Theta)$ is $M$-regular with Euler characteristic equals to $1$. 
The hypotheses of our setting grant that\linebreak $\mathrm{Supp}\,\big(a_*\OO{X}\otimes\varphi^*\OO{A}(\Theta)\big)=a(X)$ spans $B$. 
The corollary then follows by the lemma above.
\end{proof}

As a consequence of the previous discussion, the line bundle $L\simeq\widehat\varphi^*M$ is ample too, and therefore, by Remark \ref{rem:prod}, it is a principal polarization. 
We deduce immediately that $\widehat \varphi$ is an isomorphism, and thus so is $\varphi$.

\subsection{Case 2: When $X$ does not generate $A$ (proof of Theorem \ref{thmintr} (2))}\hfill\par
Suppose now that $X$ generates a proper abelian subvariety $J=\langle X \rangle\subsetneq A$. 
Then by Poincar\'{e}'s reducibility theorem
there exists an abelian subvariety $B\subset A$ such that the restriction $\psi:J\times B\rightarrow A$ 
of the multiplication map from $A$ to $J\times B$ is an isogeny such that
\begin{equation}\label{varphibox}
\psi^* \fas O_A(\Theta)\simeq \fas O_A(\Theta)_{|J}\boxtimes 
\fas O_A (\Theta)_{|B}.
\end{equation}
From the very definition of $\psi$ there is a commutative diagram
% \centerline{ 
\begin{align}\label{eq:factors}
\xymatrix@=32pt{  
& X \ar[d]^{\iota}\ar[r]^{\iota'} & J \ar[d]^{(\mathrm{id}_J, 0_B)}\\ 
& A  & J\times B\ar[l]^{\psi}}
\end{align}
\noindent where $\iota$ and $\iota'$ are the inclusion maps of $X$ in $A$ and $X$ in $J$ respectively, 
and $(\mathrm{id}_J,0_B)$ is the morphism taking $p$
to $(p,0_B)$.
Finally, we denote by $a:X\rightarrow J\times B$ the composition $(\mathrm{id}_J,0_B) \circ \iota'$.

Observe that by Proposition \ref{lem:trfonA} we have an isomorphism
$$\mathbf{R}S_A\mathbf{R}\Delta_A(\iota_*\OO{X}(\Theta))\simeq \fas I_Z\otimes L[-\dim A]$$ where $\fas I_Z$ is a sheaf of ideals such that 
$Z$ does not contain any divisorial component and $L$ is a line bundle on $\widehat A$.
Moreover, by Proposition \ref{prop:tehc}, there is an isomorphism
\begin{equation}\label{eq:nongen}
\mathbf{R}S_{J\times B}\mathbf{R}\Delta_{J\times B}(a_*\OO{X}\otimes\psi^*\OO{A}(\Theta))\simeq \fas{I}_W\otimes M[-g]
\end{equation}
where $\fas I_W$ is an ideal sheaf on $\widehat J\times \widehat B$ whose zero locus does not contain any divisorial component, and $M$ is a line bundle on
$\widehat J\times \widehat B$ such that $\widehat \psi ^* M\simeq L$.
Finally we notice that the left hand side of the above isomorphism can be written as follows:
\begin{align*}
& \mathbf{R}S_{J\times B}\mathbf{R}\Delta_{J\times B}(a_*\OO{X}\otimes\psi^*\OO{A}(\Theta))\\
&\simeq \mathbf{R}S_{J\times B}\mathbf{R}\Delta_{J\times B}(a_*\OO{X}\otimes(\OO{A}(\Theta)_{|J}\boxtimes\OO{A}(\Theta)_{|B}))\\
&\simeq  \mathbf{R}S_{J\times B}\mathbf{R}\Delta_{J\times B}(\iota'_*(\OO{X}\otimes\OO{A}(\Theta)_{|J})\boxtimes(\fas O_{0_B}\otimes\OO{A}(\Theta)_{|B}))\\
& \simeq \mathbf R S_{J}\mathbf{R}\Delta_{J}(\iota'_*\OO{X}\otimes\OO{A}(\Theta)_{|J})\boxtimes\mathbf R S_B(\mathbf R \Delta_B(\fas O_{0_B}))\\
& \simeq \mathbf R S_{J}\mathbf{R}\Delta_{J}(\iota'_*\OO{X}\otimes\OO{A}(\Theta)_{|J})\boxtimes\OO{\widehat B}[-\dim \widehat B].
\end{align*}
where $\fas O_{0_B}$ denotes the skyscraper sheaf at $0_B$.
We deduce that $$R^i S_{J}\mathbf{R}\Delta_{J}(\iota'_*\OO{X}\otimes\OO{A}(\Theta)_{|J})=0 \quad\mbox{ for every }\quad i<\dim J,$$ and that 
$R^{\dim J} S_{J}\mathbf{R}\Delta_{J}(\iota'_*\OO{X}\otimes\OO{A}(\Theta)_{|J})$ is a torsion free sheaf of rank $1$.
Therefore we can write 
\begin{align}\label{M}
\mathbf R S_{J}\mathbf{R}\Delta_{J}(\iota'_*\OO{X}\otimes\OO{A}(\Theta)_{|J})\simeq\fas{I}_{W'}\otimes M'[-\dim J]
\end{align}
where $\fas{I}_{W'}$ is an ideal sheaf such that its zero locus does not contain any divisorial components and $M'$ is a line bundle on $\widehat J$. 
Then, by combining \eqref{M} with \eqref{eq:nongen} we obtain a series of isomorphisms
$$\fas{I}_W\otimes M[-g]\simeq \big(\fas{I}_{W'}\otimes M'[-\dim J]\big)\boxtimes\OO{\widehat B}[-\dim \widehat B]\simeq \big(\fas{I}_{W'}\otimes M'\big)
\boxtimes\OO{\widehat B}[-g]$$ from which we see 
that $M\simeq M'\boxtimes \OO{\widehat B}$ since they are reflexive hulls of isomorphic sheaves.

Finally, by using Lemma \ref{lem:mample}, and by reasoning as in Corollary \ref{cor:mample}, we get the following
\begin{cor}
 The line bundle $M'$ in \eqref{M} is ample on $\widehat J$.
\end{cor}
Therefore we can conclude that 
$$ L\simeq \widehat\psi^*M\simeq \widehat\psi^*(M'\boxtimes\OO{\widehat B})$$ 
is effective but not ample. Furthermore, from Remark \ref{rem:prod} $\widehat A$ can 
be written as a nontrivial product $\widehat A\simeq \widehat A_1\times \widehat A_2$ and, 
up to switching factors, there is an isomorphism $L\simeq p_1^*(\OO{\widehat A_1}(\Theta_1))$ where $\Theta_1$ is a principal polarization on $\widehat A_1$.
However more is true. In fact, by K\"unneth formula, we obtain inequalities 
$$0<h^0(\widehat J, M')=h^0(\widehat J, M')\cdot h^0(\widehat B,\OO{\widehat B})\le h^0(\widehat A, L)=1$$ which imply 
that $M'$ is a principal polarization on $\widehat J$. Furthermore, since $L$ is obviously trivial on $\widehat\psi^{-1}(\{0_{\widehat J}\}\times\widehat B)$, 
we have the following commutative diagram:
$$
\xymatrix{\widehat A\simeq \widehat A_1\times \widehat A_2\ar[d]_{p_1}\ar[rr]^{\widehat\psi}&&\widehat J\times \widehat B\ar[d]^{q_J}\\
\widehat A_1\ar[rr]_{\pi\,\,\,\,\,\,\,\,\,\,\,\,\,\,\,\,\,\,\,\,\,\;\;\;\;\;\;\;\;\;\;\;}&&\widehat J\simeq \widehat A_1/p_1\big(\widehat  \psi^{-1}
\big(\{0_{\widehat J}\}\times \widehat B\big)\big)}
$$
Finally, by projection formula we have the following chain of equalities of algebraic sets
$$A_1\times\{0_{A_2}\}=V^0_{\widehat A_1\times \widehat A_2}(L)=\psi (V^0_{\widehat J\times \widehat B}(M'\boxtimes \OO{\widehat B}))=
\psi(J\times\{0_{B}\})$$
from which we conclude that $\pi$ is an isogeny. We deduce that $A_1$ and $J$ have the same dimension and further that
$h^0(A_1, \pi^*M')=1$ since $L\simeq p_1^*\pi^*M'$. Therefore
we conclude that $\pi$ is an isomorphism and that $\widehat J$ is a direct factor of 
$\widehat A$ such that $\OO{\widehat A}(\Theta)_{|\widehat J}$ is a principal polarization. This immediately says that 
$J$ is a direct factor of $A$ and that $\OO{A}(\Theta)_{| J}$ 
is a principal polarization on $J$.

In order to complete the proof of Theorem \ref{thmintr} we need to show that $\fas{I}_{X/J}\otimes\OO{A}(\Theta)_{|J}$ is a $GV$-sheaf. 
To this end, consider the short exact sequence
$$
0\longrightarrow \fas{I}_{X/J}\otimes\OO{A}(\Theta)_{|J}\longrightarrow \OO{A}(\Theta)_{|J} \longrightarrow \iota'_*\OO{X}\otimes \OO{A}(\Theta)_{|J}\longrightarrow 0
$$
to which we apply the functor $\mathbf{R}S_J\mathbf R\Delta_J$. By taking cohomology
we see that 
$$R^i S_J\mathbf R\Delta_J(\fas{I}_{X/J}\otimes\OO{A}(\Theta)_{|J})\simeq R^{i+1}S_J\mathbf R\Delta_J(\fas{O}_{X}\otimes\OO{A}(\Theta)_{|J})=0
\quad\mbox{ for }\quad  i<g-1.$$
In addition, we have that $$R^{g-1}S_J\mathbf R\Delta_J(\fas{I}_{X/J}\otimes\OO{A}(\Theta)_{|J})=0$$ 
since it is the kernel of a nonzero morphism between torsion free sheaves of equal rank.
We conclude by invoking Theorem \ref{thm:alg}

\section{An application}
We apply Theorem \ref{thmintr} to provide a new proof of a theorem of Clemens--Griffiths \cite{CG1972}*{Theorem 11.19 and (0.8)}.

\begin{thm}[\cite{CG1972}*{Theorem 11.19}]
Let $Y\subset\mathbb \rp^4$ be a smooth cubic threefold and let $S$ be the Fano scheme parametrizing 
lines of $\rp^4$ contained in $Y$. Moreover, denote by $\mathrm{Alb}(S)$ the Albanese variety of $S$ and by $J(Y)$ the intermediate Jacobian of $Y$.
 Then $\mathrm{Alb}(S)$ and $J(Y)$ are isomorphic.
\end{thm}
\begin{proof}
They are classical results that the irregularity $q(S)=\dim H^1(S,\fas{O}_S)$ of $S$ equals $5$, and that the map 
$\pi: \mathrm{Alb}(S)\rightarrow J(Y)$ induced by the universal property of the Albanese variety is an isogeny.
For instance, this was observed by Fano and Gherardelli over the complex numbers, and by Altman--Kleiman and Murre over any algebraically closed field of characteristic 
different from $2$ \cites{Gh1967,AK1977,Mu1974}. In particular, we notice that $S$ generates $J(Y)$ as a group.
Moreover, H\"{o}ring in \cite{Ho2007} proves that $S$ is a $GV$-subscheme in $J(Y)$ by only
exploiting the structure of a Prym variety on the intermediate Jacobian $J(Y)$.
Thus all the hypotheses of Theorem \ref{thmintr} (1) are satisfied and so $\pi$ is an isomorphism.
 \end{proof}

\bibliography{minimal.bib}

\end{document}